\documentclass[12pt]{amsart}
\usepackage{amsmath}

\def\ol{\overline}
\def\e{\epsilon}
\def\lf{\left}
\def\ri{\right}
\def\a{{\alpha}}

\def\wt{\widetilde}
\def\tn{{\wt\nabla}}
\def\p{\partial}
\newcommand\R{{\mathbb R}}

\newcommand\C{{\mathbb C}}

\def\ii{\sqrt{-1}}
\def\jbar{{\bar\jmath}}

\def\tth{\tilde{h}}

\def\K{K\"ahler }
\def\KR{K\"ahler-Ricci }
\def\KRF{K\"ahler-Ricci flow }

\def\A{Amp\`{e}re }
\def\H{H\"{o}lder }
\def\ddbar{\partial\bar\partial}
\def\be{\begin{equation}}
\def\ee{\end{equation}}
\def\ol{\overline}
\def\lf{\left}
\def\ri{\right}
\def\a{{\alpha}}

\def\e{\epsilon}

\def\ijb{{i\jbar}}

\def\Ric{\text{\rm Ric}}
\def\Rm{\text{\rm Rm}}

\def\wt{\widetilde}
\def\tn{{\wt\nabla}}
\def\p{\partial}

\def\C{\Bbb C}

\def\wt{\widetilde}
\def\tn{{\wt\nabla}}
\def\p{\partial}

\def\p{\partial}

\def\C{\Bbb C}
\def\ii{\sqrt{-1}}

\def\KRF{K\"ahler-Ricci flow }

\def\ttR{\wt  R}

\def\tn{\wt\nabla}
\def\tD{\wt\Delta}
\newtheorem{thm}{Theorem}[section]

\newtheorem{lem}{Lemma}[section]
\newtheorem{prop}{Proposition}[section]
\newtheorem{cor}{Corollary}[section]
\theoremstyle{definition}
\newtheorem{defn}{Definition}[section]
\theoremstyle{remark}
\newtheorem{rem}{Remark}
\numberwithin{equation}{section}
 \begin{document}
 \title{On a modified parabolic complex Monge-Amp\`{e}re equation with applications}
\author{Albert Chau$^1$}

\address{Department of Mathematics,
The University of British Columbia, Room 121, 1984 Mathematics
Road, Vancouver, B.C., Canada V6T 1Z2} \email{chau@math.ubc.ca}

\author{Luen-Fai Tam$^2$}

\thanks{$^1$Research
partially supported by NSERC grant no. \#327637-06}
\thanks{$^2$Research partially supported by Hong Kong RGC General Research Fund
\#GRF 2160357}

\address{The Institute of Mathematical Sciences and Department of
 Mathematics, The Chinese University of Hong Kong,
Shatin, Hong Kong, China.} \email{lftam@math.cuhk.edu.hk}
\thanks{\begin{it}2000 Mathematics Subject Classification\end{it}.  Primary 53C55, 58J35.}
\thanks{\begin{it}Key words and phrases\end{it}.  Non-compact K\"ahler-Einstein metrics, K\"ahler-Ricci flow, parabolic  Monge-Amp\`{e}re  equation.}

\begin{abstract} We study a parabolic complex Monge-\A type equation of the form \eqref{MA} on a complete non-compact \K manifold. We prove a short time existence result  and obtain basic estimates. Applying these results, we prove that under certain assumptions on a given real and closed (1,1) form $\Omega$ and initial \K metric $g_0$ on $M$, the modified \KR flow $g'=-\Ric+\Omega$ has a long time solution converging to a complete \K metric such that $\Ric=\Omega$, which extends the result in \cite{Ca} to non-compact manifolds. We will also obtain a long time existence result for the \KR flow which generalizes a result \cite{ChauTamYu08}.
\end{abstract}

 \maketitle\markboth{Albert Chau and Luen-Fai Tam} {On a modified parabolic complex Monge-Amp\`{e}re equation with applications}

\section{introduction}

 Let $(M, g_0)$ be a smooth complete non-compact \K manifold.  In this article we will study parabolic complex-Monge \A equations of the following type on $M$:

  \begin{equation}\label{MA}
\left\{%
\begin{array}{ll}
    \dfrac{\p v}{\p t}= \log \dfrac{(\sigma(t)+ \ii\partial\bar{\partial}v)^n}{(\omega_0)^n}- f \ \ \text{in $M\times[0,T)$}
   \\
   v(x,0)=u. \\
\end{array}%
\right.
\end{equation}
where $f$ and $\sigma(t)$ are given smooth families of functions and real and closed $(1, 1)$ forms on $M$ for $t\in [0, T)$ respectively, $u$ is a smooth function on $M$ and $\omega_0$ is the \K form of $g_0$.  There is a close connection between \eqref{MA} and the \KR flow equation on $M$

  \begin{equation}\label{KRF}
\left\{%
\begin{array}{ll}
    \dfrac{\p \omega }{\p t}=-\Ric(\omega)      \\
   \omega(0)=\omega_0. \\
\end{array}%
\right.
\end{equation}
where $\omega$ is a \K form with corresponding Ricci form $\Ric(\omega)$.  Let $v$ be a solution to \eqref{MA} such that   the corresponding family $\omega(t) := \sigma(t) +\ii \partial\bar{\partial}v$ are \K forms.  Then if $\Ric_0=\sqrt{-1}\partial\bar{\p}f$ and $\sigma(t)=\omega_0$, then $\omega(t)$ evolves under \eqref{KRF} and remains in the \K class $[\omega_0]$ (see \cite{Ca}, \cite{ChauTam09}).  On the other hand, if $f=0$ and $\sigma_t=-t\Ric(\omega_0)+\omega_0$, then $\omega(t)$ evolves under \eqref{KRF} but does not remain in the same \K class in general (see \cite{TZ}, \cite{Ts} on compact manifolds and \cite{LZ} on non-compact manifolds).  Conversely, it can also be shown in the above cases that given a solution $\omega(t)$ to \eqref{KRF}, there exists a corresponding solution $v$ to \eqref{MA} such that $\omega(t) := \sigma(t) +\ii \partial\bar{\partial}v$ (see references above).

  One of the goals in this article is to generalize our previous results in \cite{ChauTam09} where the authors proved: when $\Ric_0=\sqrt{-1}\partial\bar{\p}f$ and $\sigma(t)=\omega_0$, then $g_0$ converges to a \K  Ricci flat metric $g$ under \eqref{KRF} under certain additional assumptions on $f$ and $g_0$.  We will extend this result to the case when $\Ric_0-\Omega=\sqrt{-1}\partial\bar{\p}f$ where $\Omega$ is given but  not necessarily   zero.   We prove that when $(M^n,g_0)$ is complete, non-compact with bounded curvature, with volume growth $V_{x_0}(r)\le Cr^{2n}$ for some $x_0$  and $C$ for all $r$,
   and satisfies a certain Sobolev inequality, then:

 \begin{it} Under the above conditions, the \KRF \eqref{KRF} has a long
 time solution $g(t)$ converging smoothly on $M$ provided
 $|f|(x)\le \frac{C}{1+\rho_0^{2+\e}(x)}$ for
 some $C, \e>0$ and all $x$ such that the Ricci form of the limit metric is $\Omega$.\end{it}

\noindent See Theorem \ref{mainthm} for details.  This will correspond to \eqref{MA} in the case $\omega(t)=\omega_0$ and $\Ric_0-\Omega=\sqrt{-1}\partial\bar{\p}f$ in which case the equation for $\frac{\p}{\p t}\omega(t)$ will change from \eqref{KRF} only by the addition of  $\Omega=\sqrt{-1}\partial\bar{\p}f$ to the RHS of \eqref{KRF}.  The proof combines the a priori estimates developed here together with estimates from \cite{ChauTam09}, in particular the $C^0$ estimate.  The main difference here is that the corresponding metrics $g(t)$ are not evolving under the standard \KR flow \eqref{KRF}, and general \KR theory cannot be directly  applied as in  \cite{ChauTam09}.  Our results are motivated by the results in  \cite{TY2, TY3} which extend the famous results in \cite{Y} to the complete non-compact setting under additional assumptions.  By studying the elliptic Monge-\A equation Yau \cite{Y} proved that if $(M, g_0)$ is a compact \K manifold and $\Omega\in c_1(M)$, then there exists a \K metric $g$ in the same class as $g_0$ with $\Omega$ as its Ricci tensor. This result was later re-established in \cite{Ca} by considering the corresponding parabolic complex Monge-\A on compact manifolds.

Our second goal will be to establish a longtime existence result for \eqref{KRF}.  We prove that when $(M^n,g_0)$ is smooth, complete and non-compact with injectivity radius bounded below and curvature approaching  zero pointwise at infinity then

 \begin{it} Under the above conditions, the \KRF \eqref{KRF} has a smooth longtime solution provided there exists a strictly plurisubharmonic function on $M$.\end{it}

 \noindent See Corollary \ref{longtimeexistence} for details. Since a simply connected complete non-compact \K manifold with nonnegative holomorphic bisectional curvature is a product of a compact \K manifold with nonnegative holomorphic bisectional curvature and a complete non-compact \K manifold with nonnegative holomorphic bisectional curvature supporting a strictly pluri-subharmonic function \cite{NiTam2003-2}, the theorem generalizes the longtime existence result in \cite{ChauTamYu08}.  In particular, the result applies when $M=\C^n$ or more generally a Stein manifold.  We establish this by showing \eqref{MA} has a longtime solution for appropriate choices of $f$ and $\sigma(t)$.  The proof combines ideas from \cite{LZ},  the a priori estimates developed here and the results in  \cite{ChauTamYu08}. In fact, this is a corollary of a more general result Theorem \ref{singularities}.

  The paper is organized roughly as follows.  In \S 2 we prove a general shorttime existence result for \eqref{MA} where we do not assume $f$ or $v(x, 0)$ are bounded.  In \S 3 we prove a priori estimates for \eqref{MA} where we assume $f$ and $v(x, 0)=0$ are bounded.  In \S 4 we prove the main results Theorem \ref{longtimeexistence} and Theorem \ref{mainthm}.

 \section{Short time existence }\vskip .2cm

Consider parabolic complex Monge-Amp\`ere equation \eqref{MA} on a complete non-compact \K manifold  $(M, g_0)$ where $\sigma=\sigma(t)$ is a given smooth real and closed $(1, 1)$ form on $M$ for $t\in [0, T)$, $\omega_0$ is the \K form of $g_0$, and $f$ and $u$ are smooth possibly unbounded  functions on $M\times[0, T)$ and $M$ respectively.

We will do our analysis of \eqref{MA} in appropriate H\"{o}lder spaces on $M$ which in turn will involve the notions of bounded geometry of various orders with respect to $g_0$.  We will now make the appropriate definitions for these.  We begin by recalling the definition for  a complete K\"ahler manifold $(M^n, g)$ ($n$ is the complex dimension) to have bounded geometry of a certain order, and we also recall the corresponding parabolic and elliptic H\"older spaces on $M$ relative to $g_0$ (see also \cite{Chau04, TY2, TY3}).

\begin{defn}\label{boundedgeom} Let $(M^n, g)$ complete K\"ahler manifold. Let  $k\ge 1$ be an integer and $0<\alpha<1$. $g$ is said to have  bounded geometry of   order $k+\alpha$ if there are positive numbers $r, \kappa_1, \kappa_2$  such that at every $p\in M$ there is a neighborhood $U_p$ of $p$, and local biholomorphism $\xi_{p}$ from $D(r)$ onto $U_p$ with $\xi_p(0)=p$ satisfying  the following properties:
  \begin{itemize}
    \item [(i)] the pull back metric $\xi_p^*(g)$   satisfies:
    $$
    \kappa_1 g_e\le \xi_p^*(g)\le \kappa_2 g_e$$ where $g_e$ is the standard metric on $\C^n$;
    \item [(ii)] the components $g_{i\jbar}$ of $\xi_p^*(g)$ in the natural coordinate of $D(r)\subset \C^n$  are uniformly bounded in the standard $C^{k+\alpha}$ norm in $D(r)$ independent of $p$.
  \end{itemize}

 \end{defn}
 It is obvious that if $g_0$ is of bounded geometry of order $k+\alpha$, then it is of bounded geometry of order $k+\alpha'$ for all $\alpha'<\alpha$.  The  family $$\mathcal{F}=\{(\xi_p,U_p), p\in M\}$$ is called  a family of {\it quasi-coordinate neighborhoods}.

   In the following we will define a norm on $m$ forms on $M$ for any $m$ and we will define corresponding Banach spaces.  We will make these definitions relative to fixed  quasi-coordinates $\mathcal{F}$.  For two different quasi-coordinates, the corresponding norms defined will be equivalent f, and the corresponding Banach spaces will be the same. See the appendix for details.

  For any domain $\Omega$ in $\C^n$ and integer $k\ge 0$ and $0<\alpha <1$, let $||\cdot||_{\Omega,k+\alpha}$ be the standard $C^{k+\alpha}$ norm for functions on $\Omega$.  If $T>0$ and $k$ is even let $||\cdot||_{\Omega\times[0,T);k+\alpha,k/2+\alpha/2 }$ be the standard parabolic $C^{k+\alpha, k/2+\alpha/2}$ norm for functions on $\Omega\times[0,T)$ (see appendix see details).

  Define the $C^{k+\alpha}(M)$ norm for a smooth $m$-form $f$ on $M$ by
\begin{equation}\label{Cb-norm-1}
     ||f||_{m,k,\alpha}=\sup_{p\in M}\max_I||(\xi_p^*f)_{I}||_{D(r),k+\alpha}.
\end{equation}
where $I$ represents a multi-index and the $(\xi_p^*f)_{I}$'s are the local components of the form $\xi_p^*f$.  Likewise, if $T>0$ and $k$ is even define the $C^{k+\alpha,k/2+\alpha/2}(M\times[0,T))$ norm for smooth time dependent $m$-form $f$ on $M\times[0,T)$ by
\begin{equation}\label{Cb-norm-2}
     ||f||_{m,k,\alpha}=\sup_{p\in M}\max_I||(\xi_p^*f)_I||_{_{D(r)\times[0,T);k+\alpha,k/2+\alpha/2}}.
\end{equation}

\begin{defn}\label{Holderspace}
For any $m$ and $0\le k\le 2$ we define $C_{m}^{k+\alpha}(M)$ to be the  norm completion of space of smooth $m$-forms with norm $||\cdot||_{m,k,\alpha}$.  Likewise, for $k=0$ or $2$, let $C_{m}^{k+\alpha,k/2+\alpha/2}(M\times[0,T))$ be the  norm completion of space of smooth time dependent  $m$-forms with norm $||\cdot||_{m,k,\alpha}$.   Both $C_{m}^{k+\alpha}(M)$ and $C_{m}^{k+\alpha,k/2+\alpha/2}(M\times[0,T))$ are Banach spaces.   We will omit the subscript $m$ in the notation when there will be no confusion.
\end{defn}

We have the following lemma: \ref{boundedgeom}.

\begin{lem}\label{2+a-l1} Let $(M^n,g)$ be a complete non-compact
K\"ahler manifold of bounded geometry of order $2+\alpha$.  Then
\begin{enumerate}
\item [(i)] $M$ has bounded curvature.
\item [(ii)] 
  There is a smooth function $\rho\ge 1$ such that near infinity it is equivalent to the distance function   from a fixed point and has bounded gradient and Hessian.
\end{enumerate}
\end{lem}
\begin{proof} (i) is obvious from the definition of bounded geometry of order $2+\alpha$. (ii) is a result in \cite{Sh2}, see also \cite{Tam07}.
\end{proof}

In order to state the main result of  short time existence of \eqref{MA}, we  first discuss the following special case:
\begin{equation}\label{MA-special}
\left\{%
\begin{array}{ll}
    \dfrac{\p v}{\p t}= \log \dfrac{(\sigma(t)+ \ii\partial\bar{\partial}v)^n}{(\omega_0)^n}\ \       \\
   v(x,0)=0. \\
\end{array}%
\right.
\end{equation}

\begin{lem}\label{shorttime-l1} Let $(M^n,g_0)$ be a smooth complete non-compact
K\"ahler manifold of bounded geometry of order $2+\alpha$  and let $\sigma=\sigma(t)$ be a smooth family of closed and real (1,1) forms on $M\times[0,T)$ such that
 \begin{itemize}
   \item [(i)] $\sigma\in C^{\alpha,\frac\alpha2}(M\times[0,T))$;
   \item [(ii)] $c^{-1}\omega_0\le\sigma\le c\omega_0$ for some $c>0$ on $M\times[0,T)$;
   \item [(iii)] there exists $v_0\in C^{2+\alpha,1+\frac\alpha2}(M\times[0, T))$ such that $$w_0:=\frac{\p v_0}{\p t}-\log\lf(\dfrac{(\sigma+\ii\ddbar v_0)^n}{\omega_0^n}\ri)$$ satisfies $w_0(x,0)=0$.
 \end{itemize}
Then there exists $0<T'\le T$ such that \eqref{MA-special} has a smooth solution $v\in C^{2+\frac\alpha2,1+\frac\alpha4}(M\times[0,T'])$ so that  $\sigma+\ii \ddbar v$ is uniformly equivalent to $\omega_0$ in $M\times[0,T']$.
\end{lem}
\begin{proof} The idea is based on a general implicit function Theorem argument outlined in \cite{H1} \footnote{Also see Proposition 5.1 in \cite{CCH} for an application of this argument to the entire graphical Mean Curvature flow}.  By choosing a smaller $T$ if necessary, we may assume that $\sigma$ is uniformly equivalent to $\omega_0$ in $M\times[0,T]$.  In particular, there exists $\delta>0$ and $C_1>0$ such that if $\||v||_{2+\frac\alpha2,1+\frac\a4}<\delta$ then $C_1\omega_0\ge\sigma+\ii\ddbar v\ge C_1^{-1}\omega_0$ in $M\times[0,T]$.  For the remainder of the proof for any $k$ and $\beta$ we will denote the spaces  $C^{k+\beta,k/2+\beta/2}(M\times[0,T])$ and $C^{k+\beta}(M)$ simply by  $C^{k+\beta,k/2+\beta/2}$ and $C^{k+\beta}$.

We define

  $$ \mathcal{B}=\{v\in C^{2+\frac\alpha2,1+\frac\a4}| \ ||v||_{2+\frac\alpha2,1+\frac\a4}<\delta, v(x,0)=0\}.
  $$
Then $\mathcal{B}$ is an open ball in a closed subspace of $C^{2+\frac\alpha2,1+\frac\a4}$.
 Now define the map
$$
F: \mathcal{B}\to C^{\frac\alpha2,\frac\a4}
$$
by
$$
F(v)=\dfrac{\p v}{\p t}- \log \frac{(\sigma+\ii\ddbar v)^n}{\omega^n}.
$$
Then the map $F$ is well defined and $C^1$ such that the differential $DF_{v}$ at any $v\in \mathcal{B}$ is given by
$$
DF_{v}(\phi)=\dfrac{\p \phi}{\p t}- (^v\sigma)^\ijb\phi_\ijb
$$
where $(^v\sigma)^\ijb$ is the inverse of $(^v\sigma)_\ijb:=\sigma_\ijb+v_\ijb$.

\begin{bf}Claim 1:\end{bf} $DF_{v}$ is a bijection from the Banach space
$$
  \mathcal{B}_1=\{\phi\in C^{2+\frac\alpha2,1+\frac\a4}|\ \phi(x,0)=0\}
  $$
  onto $C^{ \frac\alpha2,\frac\a4}$.

 Note that the claim is straight forward in the case that $^v\sigma$ has bounded curvature on $M\times[0, T]$.  As we cannot assume this however, we must proceed more carefully.  Let $\rho$ be a smooth function on $M$ equivalent to the distance function with respect to $g_0$ from some point $p$ as Lemma \ref{2+a-l1}. Since the metrics $^v\sigma$ are uniformly equivalent to $g_0$ and $\rho\ge1$, there is a constant $C_2$ such that
  $$
  |^v\sigma_\ijb \rho_\ijb|< C_2\rho
  $$
  in $M\times[0,T]$. Hence if $\phi\in \mathcal{B}_1$, then for any $\e>0$,
  $$
  \frac{\p }{\p t}(\phi+\e e^{C_2t}\rho)-^v\sigma_\ijb(\phi_\ijb+\e e^{C_2t}\rho_\ijb)>0.
  $$
 On the other hand, the minimum of  $\phi+\e e^{C_2t}\rho$ is attained in a compact set of $M\times[0,T]$, and thus by the maximum principle, we conclude that $\phi+\e e^{C_2t}\rho\ge0$ because $\phi(x,0)=0$. Letting $\e\to0$, we conclude $\phi\ge0$ on $M\times[0, T]$. Similarly, one can prove that $\phi\le0$. Hence $\phi=0$ and $DF_v$ is injective.

 Now, let $\Omega_l$ be bounded domains with smooth boundary which exhaust $M$ and $^l\phi$ be the solution of
 $DF_{v}(^l\phi )=w$ in $\Omega_l\times[0,T ]$ where $\phi^l=0$ for $t=0$ and on $\p\Omega_l\times[0,T ]$ and $w\in C^{\frac\alpha2,\frac\a4}(M\times[0,T]$. Let $C_3>\sup_{M\times[0,T]}|w|$. Then
 $DF_{v}(^l\phi+C_3t)>0$. By the maximum principle, we conclude that $^l\phi\ge -C_3$. Similarly, we have $^l\phi\le C_3$. Hence the sequence $|^l\phi|$ is uniformly bounded by $C_3$.

 Now for any $p\in M$, let $(\xi_p,U_p)$, and $\xi_p: D(r)\to U_p$ be as in Definition \ref{boundedgeom}. The pull back of $^l\phi$ satisfies:
 $$
 \dfrac{\p \,^l\phi}{\p t}-(^v\sigma )^{\ijb} (^l\phi)_\ijb =w
$$
in $D(r)$. For simplicity, we use  $^l\phi$ to denote the pull back of $^l\phi$, etc.  By our hypothesis, the components $(^v\sigma)^{\ijb}$ above are uniformly equivalent to the standard Euclidean metric and are uniformly bounded in the standard $C^{\alpha, \frac\alpha2}$ norm on $D(r)\times[0, T]$.  Then by standard Schauder estimates we have
$$
||^l\phi||_{_{D(\frac r2)\times[0,T],2+\frac\alpha2, 1+\frac\alpha4}}\le C_4
$$
for some $C_4$ independent of $p$ and sufficiently large $l$.  Now a standard diagonalizing subsequence argument produces a $\phi \in C^{2+\frac\alpha2,1+\frac\a4}$ such that $DF_v (\phi)=w$. So $DF_v$ is surjective and the claim is established.

 Now let $v_0$ be the function in (iii) which is in $C^{2+\alpha,1+\frac\a2}(M\times[0,T))$, then $w_0=F(v_0)$ and $w_0(x,0)=0$. By the inverse function theorem, there exists $\e>0$ such that if $||w-w_0||_{\frac\alpha2, \frac\alpha4}<\e$, then there is $v\in C^{2+\frac\a2 ,1+\frac\a4}$ such that $F(v)=w$.

For any $0<\tau<1$, let $w_\tau$ be such that
\begin{equation}
w_\tau(x, t)=\left\{%
\begin{array}{ll}
 &   0, \hspace{2.1cm} t\leq \tau\\
 & w_0(x, t-\tau),\hspace{0.25cm} \tau<t<1
\end{array}%
\right.
\end{equation}
\begin{bf}Claim 2:\end{bf} $||w_\tau-w_0||_{\frac\alpha2, \frac\alpha4}<\e$ for sufficiently small $\tau>0$.

We will still use $w_\tau$ and $w_0$ to denote the respective pull backs under $\xi_p$.  Let $x, x'\in D(r)$, $t, t'$. Let $A$ be the $2+\alpha$ norm of $w_0$ and let $ \eta=w_\tau-w_0$.

Case 1: $t, t'\le \tau$. Then \begin{equation}\nonumber
\begin{split}|\eta(x,t)-\eta(x',t')|&=|w_0(x,t)-w_0(x',t')|\\
&\le A\min\{\lf(|x-x'|^\alpha +|t-t'|^\frac\alpha2\ri), t^\frac\alpha2+(t')^\frac\alpha2\}.
\end{split}
\end{equation}
where we have used the fact that $w_0|_{t=0}=0$. Thus if  $|x-x'|\ge \tau^\frac12$ then $$|\eta(x,t)-\eta(x',t')|\le A  (t^\frac\alpha2+(t')^\frac\alpha2)\le2A\tau^\frac\alpha2\le 2A\tau^\frac\alpha4\lf(|x-x'|^\frac\alpha2+|t-t'|^\frac\alpha4\ri),$$
and if $|x-x'|\le \tau^\frac12$ then
$$|\eta(x,t)-\eta(x',t')|\le A \lf(|x-x'|^\alpha +|t-t'|^\frac\alpha2\ri)\le A\tau^\frac\alpha4\lf(|x-x'|^\frac\alpha2+|t-t'|^\frac\alpha4\ri)
$$
because $|t-t'|\le \tau$.  In either case above we see that the claim is true.

Case 2: $t, t'\ge \tau$. Then

\begin{equation}
  \begin{split}
    |\eta(x,t)-\eta(x',t')| &= |w_0(x,t)-w_0(x',t')-w_0(x,t-\tau)+w_0(x',t'-\tau)| \\
      & \le 2A\min\{\tau^\frac\alpha2, |x-x'|^\alpha+|t-t'|^\frac\alpha2\}
  \end{split}
\end{equation}
Thus if $|t-t'|\ge \tau$, then
$$
|\eta(x,t)-\eta(x',t')| \le 2A\tau^\frac\alpha2\le 2A\tau^\frac\alpha4\lf(|x-x'|^\frac\alpha2+|t-t'|^\frac\alpha4\ri),
$$
and if $|t-t'|\le \tau$, then we can prove as in Case 1 that
$$
|\eta(x,t)-\eta(x',t')|\le 2A\tau^\frac\alpha4\lf(|x-x'|^\frac\alpha2+|t-t'|^\frac\alpha4\ri).
$$
In either case above we see that the claim is true.

Hence by the inverse function theorem we have $F(v)=0$ on $M\times[0, \tau]$ for sufficiently small $\tau$.  In particular $v$ solves \eqref{MA-special} and satisfies the conditions in the lemma $M\times[0, \tau]$ . The fact that $v$ is smooth follows from a standard bootstrapping argument applied to \eqref{MA-special} as at the end of the proof of Proposition \eqref{MA-shorttime-l1}.
\end{proof}

\begin{prop}\label{MA-shorttime-l1} Let $(M^n,g_0)$ be a smooth complete non-compact
K\"ahler manifold of bounded geometry of order $2+\alpha$. Let $u$ be a smooth function on $M$ and let $f$ and $\sigma=\sigma(t)$ be a smooth family of functions and real and closed (1,1) forms, respectively, on $M$ for $t\in [0, T)$.  Suppose that
\begin{itemize}
  \item [(i)] $ (\|du\|_{C^{1+\alpha}(M)}+|f_t|+\sup_{t\in[0, T)}\|df(t)\|_{C^{1+\alpha}(M)})<\infty$;
\item[(ii)] $\sigma \in C^{2+\alpha,1+\frac\a2}(M\times[0, T))$;
     \item [(iii)] $\sigma+\ii\partial\bar{\partial}u\ge c\omega_0$ on $M\times[0,T)$ for some $c>0$;
 \item[(iv)] $\log\dfrac{(\sigma+ \partial\bar{\partial}u)^n}{\omega_0^n}$ is in $C^{2+\alpha,1+\frac\a2}(M\times[0, T))$.
\end{itemize}
Then there exists $0<T'\leq T$ and a smooth solution $v$ to \eqref{MA} on $M\times[0, T']$ such that $(v-u+\int_0^t f(s)ds)\in C^{2+\alpha, 1+\alpha/2}(M\times[0, T'])$ and $\sigma+\ii\partial\bar{\partial}v$ are uniformly equivalent to $\omega_0$ in $M\times[0,T']$.
\end{prop}
\begin{proof} Let $u$, $\sigma$ and $f$ be as in the Proposition and let

$$
\tilde \sigma_\ijb=\sigma_\ijb+(u-\int_0^tf(s)ds)_\ijb.$$  Then $\tilde\sigma\in C^{ \alpha, \frac\a2}(M\times[0, T'])$ and there is  $T'\in (0, T)$ and $c>0$ such that

$$
c\omega_0\ge\tilde \sigma\ge c\omega_0
$$
in $M\times[0,T']$. Let
$$
v_0=t\log\lf(\frac{(\tilde{\sigma}+\ii\ddbar u)^n}{\omega_0^n}\ri).
$$
Then $v_0$ satisfies condition (iii) in Lemma \ref{shorttime-l1} with $\sigma$ replaced by $\tilde \sigma$.
By Lemma \ref{shorttime-l1}, by choose a smaller $T'$ we may assume that there is $\tilde v\in C^{2+\frac \a2 ,1+\frac\a4}(M\times[0, T']$ solving \eqref{MA-special} (with $\sigma$ replaced by $\tilde \sigma$) such that $\tilde \sigma+\ii\ddbar v$ is uniformly equivalent to $\omega_0$ in $M\times[0,T']$.

Let $v=\tilde v+u-\int_0^tf(s)ds$. Then
$$
\sigma+\ii\ddbar v=\sigma+\ii\ddbar\tilde v+\ii\ddbar(u-\int_0^tf(s)ds)=\tilde\sigma+\ii\ddbar\tilde v
$$
and
$$
\frac{\p v}{\p t}=\frac{\p \tilde v}{\p t}-f.
$$
Hence $v$ is a solution to \eqref{MA}, $v$ is smooth and $\sigma+\ii\ddbar v$ is uniformly equivalent to $\omega_0$ in $M\times[0, T']$. It remains to prove that $\tilde v$ is actually in $C^{2+\a ,1+\frac\a2}(M\times[0, T'])$.

Around any $p\in M$,   pull the equation \eqref{MA} back into $D(r)$ by $\xi_p$ and differentiate the pullback equation with respect to $z^l$ say.   We then obtain the following in $D(r)\times[0,T]$
\begin{equation}\label{diffMA2}
\left\{%
\begin{array}{ll}

    \dfrac{\p v_l}{\p t}= (^v\sigma))^{i\jbar}(v_l)_{i\jbar} +(^v \sigma)^{i\jbar}(\sigma)_{i\jbar l}-
    (g_0)^{i\jbar}(g_0)_{i\jbar l}-f_l\\

   v_l(x,0)=u_l. \\
\end{array}%
\right.
\end{equation}
Thus \eqref{diffMA2} is a strictly parabolic equation for $v_l$ with initial condition $u_l$ being in $C^{3+\alpha}(D(\frac 34r))$  (by condition (i), (ii) and (iv)) such that $||u_l||_{3+\alpha,D(\frac34 r )}$ is bounded by a constant independent of $p$. Moreover, the $C^{\frac\a2,\frac\a4}$ norms of $(^v\sigma)^{i\jbar}$ and $(^v \sigma)^{i\jbar}(\sigma)_{i\jbar l}-
    (g_0)^{i\jbar}(g_0)_{i\jbar l}-f_l$ in $D(r)\times[0,T']$ are bounded by a constant independent of $t$.  It follows from parabolic Schauder theory  that the $C^{2+\frac\a2, 1+ \frac\alpha4}$ norm of $v_l$ in $(D(r/2)\times[0, T'])$ is bounded  by a constant independent of $p$.   Repeating the above argument with respect to a conjugate coordinate $z_{\bar{l}}$, we conclude that $C^{2+\frac\a2, 1+ \frac\alpha4}$ norms of   the first space derivatives of $v$ in $ D(\frac r2)\times[0, T'] $ are bounded by a constant independent of $p$. This implies in particular, that the $C^{\a,\frac\a2}$ norms of $(^v\sigma))^{i\jbar}$ and $(^v \sigma)^{i\jbar}(\sigma)_{i\jbar l}-
    (g_0)^{i\jbar}(g_0)_{i\jbar l}-f_l$ in $D(\frac r2)\times[0,T']$ are bounded by a constant independent of $t$ (see remark \ref{remfourthorder}). Repeating the above arguments, together with the fact that $\tilde v$ is uniformly bounded and the assumptions on $u$ and $f$ we conclude that $\tilde v$ is in the $C^{2+\a,1+\frac\a2}(M\times[0,T'])$.  By the bootstrapping argument above it is also not hard to see that $v$ is in fact smooth.
\end{proof}
\begin{rem}\label{remfourthorder}
In the second to last sentence in the proof above  we actually have that the $C^{\frac\a2,\frac\a4}$ norms of the first space derivatives of $(^v\sigma))^{i\jbar}$ and $(^v \sigma)^{i\jbar}(\sigma)_{i\jbar l}-
    (g_0)^{i\jbar}(g_0)_{i\jbar l}-f_l$ in $D(\frac r2)\times[0,T']$ are bounded by a constant independent of $t$.  Thus by the above argument  we in fact have that the $C^{2+\frac\a2, 1+ \frac\alpha4}$ norms of   the second space derivatives of $v$ in $ D(\frac r4)\times[0, T'] $ are bounded by a constant independent of $p$.
\end{rem}

 The shorttime existence of \eqref{KRF} was proved by Shi \cite{Sh1, Sh2} assuming $(M^n,g_0)$ is a  complete non-compact K\"ahler manifold with bounded curvature.  Proposition \ref{MA-shorttime-l1} re-establishes this fact under the additional assumption  of bounded geometry of order $2+\alpha$ such that the Ricci form is in $C^{2+\alpha}(M)$, with much shorter proof.  In particular, we have

 \begin{cor}
Let  $(M^n,g_0)$ be a complete non-compact K\"ahler manifold with bounded geometry of order $2+\alpha$ and the Ricci form is in $C^{2+\alpha}(M)$ for some $\alpha >0$.  Then there exists $T>0$ such that  \eqref{KRF} has a solution $g(t)$ on $M\times[0, T)$ such that for each $t$, $g(t)$ has bounded curvature and is equivalent to $g_0$.
\end{cor}\begin{proof} Apply  the Proposition to find a short time solution $v$ of \eqref{MA} with $\sigma(t)=-t\Ric_0+\omega_0$ and $f=0$.  Then $\omega(t)=\sigma(t)+\ii\ddbar v$ is the required solution  of the \KR flow.

\end{proof}

\section{A priori estimates}
Let $(M^n,g_0)$ be a smooth complete non-compact \K manifold with
bounded geometry of order $2+\alpha$ for some $0<\alpha<1$.
Let  $f$ and $\sigma=\sigma(t)$  be a family of smooth functions and \K forms
on $M\times[0, T)$ respectively.  Let
$f $ be a smooth
function on $M\times[0, T)$.  Let $v(x,t)$ be a smooth solution to the following
equation on $M\times[0,T)$
  \begin{equation}\label{MA1}
\left\{%
\begin{array}{ll}
    \dfrac{\p v}{\p t}= \log \dfrac{(\sigma+ \ii\partial\bar{\partial}v)^n}{ \omega_0^n}- f\ \ \text{in $M\times[0,T]$}
      \\
   v(x,0)=0  \\
\end{array}%
\right.
\end{equation}
where $\omega_0$ is the \K form of $g_0$.

 In this section, we want to obtain estimates on $v$.
 The derivations of the estimates are rather standard and along similar lines
 as in \cite{Y, Ca} (compact case) \cite {Chau04, ChauTam09, LZ} (non-compact case), except that we have a more simple proof for the second space derivatives of $v_t$ (see Lemma 3.4).

   Throughout the  section we will let $g_{k \bar{l}}(t)= (\sigma(t))_{k \bar{l}}+ v_{k\bar{l}}$.  We will use $\Delta$, $\nabla$, $|\cdot|$  and $\Delta_{\sigma}$, $\nabla_{\sigma}$, $|\cdot|_\sigma$ to denote the Laplacian , covariant derivatives and norms with respect to $g$ and $\sigma$ respectively.

  Let us recall some well-known results, see  \cite{Y,Chau04}.

\begin{lem}\label{higherorder-l1}
Let $h(t)$ be a smooth family of uniformly equivalent
complete \K  metrics on $M$ for $t\in[0,T)$ and let $u$ be
a smooth function on $M\times[0,T)$  such that  $\wt
h_\ijb=h_\ijb+u_\ijb$ is a family of complete \K metrics
uniformly equivalent to $h$ for all $t$.  In the following,
$\wt\nabla$, $\wt\Delta$ and $\nabla_h$, $\Delta_h$ are covariant derivatives and Laplacians  with respect to $\wt h$ and $h$ respectively.  Also, $|\cdot|$ will denote a norm with respect to $h$.
\begin{itemize}
  \item [(i)]
  \begin{equation}\label{Laplacian-e1}
    \tD\lf(\Delta_h u+n\ri)\ge \frac{\lf|\tn(\Delta_h u+n)\ri|^2}{(\Delta_h u+n )}-C(\Delta_h u+n )\tth^{\ijb}h_{\ijb}-h^{\ijb}\ttR_\ijb
  \end{equation}
where $C$ is a constant depending only on $n$ and a bound
on  the holomorphic bisectional curvature of $h$ and
$\ttR_\ijb$ is the Ricci curvature of $\tth$.
  \item[(ii)] Let $Q=\wt h^{i\bar j}\wt h^{k\bar l}\wt h^{m\bar n}u_{;i\bar
lm}u_{;\bar jk\bar n} $, where $;$ is the covariant
derivative with respect to $h$ and let
$$
F=\log\frac{\det(\wt h_{a\bar b})}{\det(h_{a\bar b})}.
$$
Then in normal coordinate with respect to $h$:
\begin{equation}\label{Q13}
  \begin{split}
  (\wt\Delta -\frac{\p}{\p t})Q\ge& (F_{;i\bar k m}-u_{t;i\bar
km})u_{;\bar ik\bar m}+(F_{;\bar i k \bar m}-u_{t ;\bar ik\bar  m})u_{;
i\bar k  m}\\
& +(u_{t;p\bar k}- F_{p\bar k})u_{;i\bar pm}u_{;\bar ik\bar m}
+(u_{t;p\bar i}- F_{p\bar i}) u_{;i\bar k m}u_{;\bar pk\bar m}
\\
& +|u_{;i\bar k m\bar a}- u_{;\bar kp\bar a}u_{;i\bar pm}|^2\\
& +|u_{;i\bar k ma}- \lf(u_{;i\bar pa}u_{;p\bar km}+u_{;m\bar
p a}u_{;i\bar kp}\ri)|^2 \\
& -C_1(n)\lf[\lf(|\nabla_h \Rm_h|+|\Rm_h|\ri)|u_\ijb|\,|u_{i\bar j k}| )+|\Rm_h|\,|u_{i\bar j k}|^2\ri]\\
& -C_2 |h_t||u_{i\bar j k}|^2
\end{split}
\end{equation}
where $\Rm_h$ is
the curvature tensor of $h$ and the last constant $C_2$
depends only on   the equivalence of $\tth$ and $h$ .
\end{itemize}
\end{lem}
 \begin{lem}\label{Deltabound1} Suppose $\Delta_\sigma v$ is bounded on $M\times[0,T']$ for all $T'<T$ and suppose there exist   $A_i$   such that in
 $M\times[0,T)$:
  $$ |v_t|\le A_1, \ |v|\le A_2,  |f|\le A_3,\ |\Delta_\sigma f|\le A_4,
    A_5^{-1}g_0\le \sigma\le A_5 g_0,\
     \ |\sigma_t|_\sigma\le A_6,
     $$
  the holomorphic bisectional curvature
  of $\sigma$ is bounded by $A_7$, and the Ricci curvature $R^0_\ijb$ of $g_0$ is bounded by
  $A_8$.

  Then there is a  positive constant  $C$
  depending only on $n$ and $A_1-A_8$  such that
  \begin{equation}\label{Deltabound2}
 C^{-1}
    \le n+\Delta_{\sigma} v\le C
\end{equation}
on $M\times[0, T)$
\end{lem}
\begin{proof}
Let $w=-v_t$. We have
\begin{equation}\label{dubound-1}
\begin{split}
    n+\Delta_{\sigma} v=&{\sigma}^{\ijb}(({\sigma})_\ijb+v_\ijb)\ge \lf(\frac{\det
    ({\sigma}_\ijb+v_\ijb)}{\det({g_0})_\ijb}\ri)^\frac1n\\
    =&\exp\lf(\frac1n(f+v_t))\ri)\\
    \ge&
    C_0
    \end{split}
\end{equation}
where $C_0$ depends only on $n$ and $A_1$ and $A_3$.  From this the
first inequality in \eqref{Deltabound2} is true. On the
other hand, by Lemma \ref{higherorder-l1} at a point with
normal coordinate with respect to $\sigma(t)$ such that
$g_\ijb(t)=(\sigma(t))_\ijb+v_\ijb=\delta_{ij}(1+v_\ijb)$
is diagonal, we have
\begin{equation}
    \begin{split}
        \lf(  \Delta -\frac{\p}{\p t}\ri)&\lf[\log(\Delta_{\sigma} v+n) \ri]
         \\
        \ge& -C_1\sum_{i}\frac1{1+v_{i\bar i}}-\frac{\sigma^{\ijb}
        R_\ijb}{ \Delta_{\sigma} v+n}-\frac{\frac{\p}{\p t}\Delta_{\sigma} v}{\Delta_{\sigma} v+n}\\
       =&-C_1\sum_{i}\frac1{1+v_{i\bar i}}
        -\frac{\sigma^{i\jbar}R^0_{i\jbar}-\Delta_{\sigma}f -(\sigma_t^{\ijb})v_{\ijb}}{\Delta_{\sigma}v+n}\\
       \ge &-C_1\sum_{i}\frac1{1+v_{i\bar i}}-C_2
     \end{split}
\end{equation}
Here $C_1$ is a constant depending only on $n$ and $A_7$
and $C_2$ is a constant depending only on $n$, $A_1$,
$A_3$, $A_4$, $A_5$, $A_7$, $A_8$  where we have   used
\eqref{dubound-1}, the fact that $u_{i\bar{i}}>-1$ for
each $i$ and the fact that
$$
\Delta_\sigma
v_t=-\sigma^{\ijb}R_\ijb+\sigma^{\ijb}R^0_\ijb-\Delta_\sigma
f
$$

by  \eqref{MA1}.  Hence for any
$0\le t\le T'$
\begin{equation}\label{Deltabound3}
    \begin{split}
        &\lf(  \Delta -\frac{\p}{\p t}\ri)\lf[\log(\Delta_{\sigma} v+n)-(C_1+1)v \ri]
         \\
         &\ge -C_1\sum_{i}\frac1{1+v_{i\bar i}}-C_2-(C_1+1)\sum_{i}\frac{v_{i\bar i}}{1+v_{i\bar i}}+(C_1+1)v_t\\
         &=  \sum_{i}\frac1{1+v_{i\bar i}}-(C_1+1)n-C_2+(C_1+1)v_t\\
          &\ge\lf(\frac{\sum_i 1+v_{i\bar i}}{\prod_i(1+v_{i\bar i})}\ri)^\frac{1}{n-1}-(C_1+1)n-C_2+(C_1+1)v_t\\
         &\ge C_3\exp\lf[\frac1{n-1}\lf(\log(\Delta_{\sigma} v+n)-(C_1+1)v\ri)\ri]-C_4\\
          &\ge C_5\lf(\log(\Delta_{\sigma}u+n)-(C_1+1)v -C_6\ri)
     \end{split}
\end{equation}
where $C_3- C_6$ are positive constants depending only on $n$ and $A_1-A_8$, where we have used \eqref{dubound-1}.

  We now want to apply a maximum principle argument to \eqref{Deltabound3}.  The argument is basically similar to that in \cite{Sh2} where maximum principles were derived for the case where $g(t)$ is evolving by \KR flow, except that in our case we do not assume the curvature of $g$ is bounded.  Since $g_0$ has bounded curvature, we can find a function $\phi$ as in Lemma \label{2+a-l1} (ii). At a point we can find holomorphic coordinates such that $g_{\ijb}$ is diagonalized at this point with respect to $g_0$. Combining \eqref{dubound-1} with our assumption that $\Delta_{\sigma} v$ is a bounded function for each $t<T$, it follows that $g(t)$ and $g_0$ are uniformly equivalent for each $t<T$.  Then $\Delta \phi=g^{\ijb}\phi_{ij}=g^{i\bar i}\phi_{i\bar i}$ is bounded in $M\times[0,T']$ for all $T'<T$. Hence $\Delta\phi\le \alpha \phi$, where $\alpha$ may depending on $T$. Consider $h=e^{\alpha t}\phi$ we have
$$
 (\Delta-\frac{\p}{\p t})h\le 0.
 $$

Consider the function
$$H
=  \lf(\log(\Delta_{\sigma}u+n)-(C_1+1) -C_6\ri)-\e h
$$ for $\e>0$  Then

$$
 (\Delta-\frac{\p}{\p t})H\ge C_5\lf(\log(\Delta_{\sigma}v+n)-(C_1+1)v+  -C_6\ri)
 $$
for $T\ge t>0$.
On $M\times[0,T]$, $H$ is bounded from above and will tend to $-\infty$ at infinity by \eqref{dubound-1} and the fact that $v$ is bounded.   $H$ must attains its  maximum. Hence by the maximum principle, $H$ cannot attain positive maximum at $t>0$. Hence
$$
\sup_{M\times[0,T]}H\le\max\{ \log n-C_6,0\}.
$$
The second inequality in \eqref{Deltabound3} is true by letting $\e\to0$.
\end{proof}
\begin{cor}\label{eqivalent1} Assume the hypothesis and notation in Lemma \ref{Deltabound1}. Then here exists a positive
constants $C>0$ depending on $n$ and $A_1-A_8$ in Lemma \ref{Deltabound1}  such that
\begin{equation}\label{equivalent2}
    C^{-1}  g_0\le g\le C  g_0,\text{ and } |v_\ijb|_\sigma\le C
\end{equation}
on $M\times[0,T)$.
\end{cor}
\begin{lem}\label{3rdorder1-l1} Assume the hypothesis and notation in Lemma \ref{Deltabound1}.
In addition assume there exist  $A_9$ and $A_{10}$   such that on $M\times[0,T)$,
$$
|\nabla_\sigma^2f|_\sigma+|\nabla_\sigma^3f|_\sigma\le A_9, \ |\Rm^\sigma|+|\nabla_\sigma\Rm^\sigma|\le A_{10},
$$
where $\Rm^\sigma$ is the curvature tensor of $\sigma$.  Suppose that $|v_{;\ijb k}|_\sigma$ is bounded on $M\times[0,T']$ for all $T'<T$.

Then   there is a constant $C$ depending only $n$ and $A_1-A_{10}$   such that

\begin{equation}\label{3rdorder1}
Q=  g^{i\bar j}   g^{k\bar l}   g^{m\bar n} v_{;i\bar
lm}v_{;\bar jk\bar n}\le C
\end{equation}
on $M\times [0,T)$.
\end{lem}
\begin{proof} In the following $C_i$'s denote positive constants depending only on $n$ and $A_1-A_{10}$.  For $0\le t<T$, by Lemma \ref{higherorder-l1} we have that in
normal coordinates with respect to $\sigma$,
\begin{equation}\label{3rdorder2}
\begin{split}
   \big( \Delta -&\frac{\p }{\p t}\big)Q\\
    \ge& (F_{;i\bar k m}-v_{t;i\bar
km})v_{;\bar ik\bar m}+(F_{;\bar i k \bar m}-v_{t ;\bar ik\bar  m})v_{;
i\bar k  m}\\
& +(v_{t;p\bar k}- F_{p\bar k})v_{;i\bar pm}v_{;\bar ik\bar m}
+(v_{t;p\bar i}- F_{p\bar i}) v_{;i\bar k m}v_{;\bar pk\bar m}
\\
& +|v_{;i\bar k m\bar a}- v_{;\bar kp\bar a}v_{;i\bar pm}|^2_\sigma\\
& +|v_{;i\bar k ma}- \lf(v_{;i\bar pa}v_{;p\bar km}+v_{;m\bar
p a}v_{;i\bar kp}\ri)|^2_\sigma \\
& -C(n)\lf[\lf(|\nabla_\sigma \Rm^{\sigma }|_\sigma+|\Rm^{\sigma }|_\sigma\ri)|v_\ijb|_\sigma\,|v_{i\bar j k}|_\sigma )+|\Rm^{\sigma}|\,|v_{i\bar j k}|^2_\sigma\ri]\\
& +C_0|v_{i\bar j k}|^2_\sigma\\
\ge& -C_1Q-C_2  +|v_{;i\bar k m\bar a}- v_{;\bar kp\bar a}v_{;i\bar pm}|^2_\sigma\\
& +|v_{;i\bar k ma}- \lf(v_{;i\bar pa}v_{;p\bar km}+v_{;m\bar
p a}v_{;i\bar kp}\ri)|^2_\sigma.
\end{split}
\end{equation}
On the other hand, direct computations show:
\begin{equation}\label{3rdorder2}
\begin{split}
   \big( \Delta -&\frac{\p }{\p t}\big)\lf(\Delta_{\sigma}v+n\ri)\\
    =& g^{k\bar l}g^{p\bar q}v_{;p\bar i l}v_{i\bar q\bar k}-g^{k\bar l}g_\ijb  R^\sigma_{i\bar jk\bar l}-\sigma^{\ijb} R_\ijb-\Delta_\sigma v_t-(\sigma^\ijb)_tv_\ijb\\
    = &g^{k\bar l}g^{p\bar q}v_{;p\bar i l}v_{i\bar q\bar k}-g^{k\bar l}g_\ijb  R^\sigma_{i\bar jk\bar l}-\sigma^{\ijb} R^0_\ijb+\Delta_\sigma f-(\sigma^\ijb)_tv_\ijb\\
    \ge &C_3Q-C_4
\end{split}
\end{equation}
where we have used \eqref{MA1} and Corollary \ref{eqivalent1}. Hence by Lemma \ref{Deltabound1}
\begin{equation}\label{thirdorderbound}
\big( \Delta -\frac{\p }{\p t}\big)\lf(Q+ \frac{C_1+1}{C_3} \lf(\Delta_{\sigma}v+n\ri)\ri)\ge Q+\frac{C_1+1}{C_3} \lf(\Delta_{\sigma}v+n\ri)-C_5.
\end{equation}
By assumption and Corollary \ref{eqivalent1}, $Q+\frac{C_1+1}{C_3} \lf(\Delta_{\sigma}u+n\ri)$ is bounded on $M\times[0,T']$ for all $T'<T$. As in the proof of Lemma \ref{Deltabound1}, we can conclude that the lemma is true by the maximum principle.
\end{proof}
Now we want to estimate the function $w=-v_t$.   We begin by noting the following estimate which follows immediately from Lemma \ref{3rdorder1-l1}, \eqref{MA1} and \eqref{equivalent2}:
\begin{cor}\label{gradientf}  With the same assumptions and notations as in Lemma \ref{3rdorder1-l1}.  In addition, suppose there exist  $A_{11}$ and $A_{12}$   such that on $M\times[0,T)$
$$
|\nabla_0\sigma|_{\sigma}\le A_{11},\ |\nabla_\sigma f|_\sigma\le A_{12}
$$
where $\nabla_0$ is the covariant derivative with respect to $g_0$.
Then there exists a positive
constants $C>0$ depending on $n$ and  and $A_1-A_{12}$  such that $|\nabla w|^2 \leq C$ on $M\times [0,T)$.
\end{cor}

Now we want to estimate higher order derivatives ow $w$.  By \eqref{MA1} we note that

\begin{equation}\label{f-e1}
\lf(\Delta -\frac{\p}{\p t}\ri)w= g^{i\jbar}   (\sigma_t)_{i\jbar}-f_t:=F.
\end{equation}
and
\begin{equation}\label{f-e2}
w_\ijb=R_\ijb-R^0_\ijb+f_{\ijb}=R_\ijb-\Omega_\ijb,
\end{equation}
where $\Omega_\ijb= R^0_\ijb-f_\ijb$.
Now we compute $\lf(\Delta-\frac{\p}{\p t}\ri)|\nabla w|^2$. In normal coordinates with respect to $g$  we have

\begin{equation}\label{f-e3}
\begin{split}
   \Delta (|\nabla w|^2) & =g^{\ijb}\lf(w_k w_{\bar l} g^{k\bar l}\ri)_{\ijb}\\
     &=g^{\ijb}g^{k\bar l}\lf(w_{ki}w_{\bar l\bar j} +w_{k\bar j}w_{\bar l i} +w_{ki\bar j}w_{\bar l}+w_kw_{\bar l\ijb}\ri)+g^{\ijb} w_k w_{\bar l} \lf(g^{k\bar l}\ri)_{\ijb}\\
     &=\sum_{i,j}\lf(|w_{ij}|^2+|w_{\ijb}|^2\ri)+(\Delta w)_i w_{\bar i}+(\Delta w)_{\bar i} w_i+R_{\ijb}w_iw_{\bar j}\\
     &=\sum_{i,j}\lf(|w_{ij}|^2+|w_{\ijb}|^2\ri)+(\Delta w)_i w_{\bar i}+(\Delta w)_{\bar i} w_i+w_{\ijb}w_iw_{\bar j}+\Omega_\ijb w_i w_{\bar j}
 \end{split}
 \end{equation}
 where we have used \eqref{f-e2}.

 \begin{equation}\label{f-e4}
\begin{split}
   \frac{\p }{\p t} (|\nabla w|^2) & =g^{\ijb}\lf(  w_{it}w_{\bar j}+w_iw_{\bar j t}\ri)+(g^{\ijb})_tw_i w_{\bar j}\\
   &=(w_t)_i w_{\bar i}+(w_t)_{\bar i}w_i +(w_{\ijb}- \sigma_{i\jbar}) w_iw_{\bar j}\\
 \end{split}
\end{equation}
Hence
\begin{equation}\label{f-e5}
\begin{split}
   \lf(\Delta-\frac{\p }{\p t}\ri) (|\nabla w|^2)   =&\sum_{i,j}\lf(|w_{ij}|^2+|w_{\ijb}|^2\ri)+(\Omega_\ijb +\sigma_{i\jbar}w_i w_{\bar j})\\
   &+F_iw_{\bar{i}}+F_{\bar{i}}w_i.
 \end{split}
\end{equation}

\begin{lem}\label{S-e6}  Assume the hypothesis and notation in Corollary \ref{gradientf}.   Assume in addition that there exist $A_{13}$ and $A_{14}$ such that on $M\times[0,T)$,
$$
|\nabla_0\sigma_t|_\sigma+|\nabla^2_0\sigma_t|_\sigma\le A_{13},\ |(f_t)_\ijb|_\sigma\le A_{14}.
$$
Suppose that $S:=g^{i\bar l}g^{k\bar j}w_{\ijb}w_{k\bar l} $ is bounded on $M\times[0,T']$ for all $T'<T$.

Then there is a constant $C$ depending only on $n$ and $A_1-A_{14}$, such that $S\le C$ on $M\times[0,T).$
\end{lem}
\begin{proof}
 In the following $C_i$ will denote positive constants depending only on $n$ and $A_1-A$.  By \eqref{f-e1} and   Lemma 2.1 in \cite{NiTam2003}
\begin{equation}\label{S2}
\begin{split}
(\Delta - \frac{\p }{\p t})w_\ijb&=  R_{l\bar k\ijb}w_{k\bar
l}-\frac12(  R_{i\bar p}w_{p\bar j}+  R_{p\bar j}w_{i\bar
l})+F_{i\jbar}\\
&= R_{l\bar k\ijb}w_{k\bar l}-\frac12(w_{i\bar p}w_{p\bar
j}+w_{p\bar j}w_{i\bar l})-\frac12(\Omega_{i\bar p}w_{p\bar
j}+\Omega_{p\bar j}w_{i\bar l})+F_{i\jbar}\\
\end{split}
\end{equation}
where we have used \eqref{f-e2}.
In normal coordinates with respect to $g$:
\begin{equation}\label{S3}
\begin{split}
(\Delta - &\frac{\p }{\p t})S\\
&=w_{\ijb k}w_{\bar i j\bar
k}+w_{\bar j i\bar k}w_{  j \bar ik}+w_{j\bar i} \Delta
w_{i\bar
j}+w_{i\bar j} \Delta w_{j\bar i}\\
&\quad -w_{j\bar i}\frac{\p}{\p t}w_{\ijb}- w_{i\bar
j}\frac{\p}{\p t}w_{j\bar i} -g^{i\bar r}g^{s\bar{l}}g^{k\bar j}w_{s\bar{r}}w_{\ijb}w_{k\bar l}-g^{i\bar l}g^{k \bar{r}}g^{s\jbar}w_{s\bar{r}}w_{\ijb}w_{k\bar l}\\
&=2|w_{i\bar jk}|^2+2w_{\ijb}\lf(  R_{l\bar k\ijb}w_{k\bar
l}-\frac12(w_{i\bar p}w_{p\bar j}+w_{p\bar j}w_{i\bar l})\ri)\\
&\hspace{12pt} -w_\ijb(\Omega_{i\bar p}w_{p\bar
j}+\Omega_{p\bar j}w_{i\bar l}+2F_{i\jbar})-g^{i\bar r}g^{s\bar{l}}g^{k\bar j}w_{s\bar{r}}w_{\ijb}w_{k\bar l}-g^{i\bar l}g^{k \bar{r}}g^{s\jbar}w_{s\bar{r}}w_{\ijb}w_{k\bar l}.
\end{split}
\end{equation}
Hence
\begin{equation}\label{S4}
\begin{split}
&(\Delta - \frac{\p }{\p
t})(1+S)^\frac12\\
&=\frac1{2(1+S)^\frac12}(\Delta - \frac{\p }{\p
t})S -\frac{|\nabla S|^2}{(1+S)^\frac32}\\
&\ge \frac1{2(1+S)^\frac12}\bigg[2w_{\ijb}\lf(  R_{l\bar k\ijb}w_{k\bar
l}-\frac12(w_{i\bar p}w_{p\bar j}+kw_{p\bar j}w_{i\bar l})\ri)\\
&\hspace{12pt}-w_\ijb(\Omega_{i\bar p}w_{p\bar
j}+\Omega_{p\bar j}w_{i\bar l}+2F_{i\jbar}+w_{i\jbar})\bigg]\\
&\hspace{12pt} -\frac1{2(1+S)^\frac12}\lf[g^{i\bar r}g^{s\bar{l}}g^{k\bar j}w_{s\bar{r}}w_{\ijb}w_{k\bar l}-g^{i\bar l}g^{k \bar{r}}g^{s\jbar}w_{s\bar{r}}w_{\ijb}w_{k\bar l}\ri]\\
&\ge - C_1(S+|\Rm|^2+|F_\ijb|^2).\\
\end{split}
\end{equation}
for some constant $C_1$ depending only on $n$ and $A_1-A$, where we have used   the fact that $|\nabla S|^2\le
2S\sum_{i,j,k}|w_{i\bar jk}|^2$.

In order to  estimate $|F_\ijb|$, note that in normal coordinate of $\sigma$ at a point,
\begin{equation}\label{curvature1}\begin{split}
     R_{i\bar j k\bar l}=g^{p\bar q}\frac{\p g_{p\bar j}}{\p z_l}\frac{\p g_{i\bar q}}{\p \bar z_k}+R^\sigma_{i\bar jk\bar l} -\frac{\p^2}{\p z_k\p \bar z_l} v_\ijb.
\end{split}
\end{equation}
Using Corollary \ref{eqivalent1}, Lemma \ref{3rdorder1-l1} we have
\begin{equation}\label{F-bound-e1}
   |\frac{\p^2}{\p z_k\p \bar z_l} v_\ijb|\le C_2(|\Rm|+1)
\end{equation}
for some constant $C_2$ depending only on the constants in the lemma. Now in a normal coordinate of $\sigma$ we have
\begin{equation*}
\begin{split}
   F_\ijb=&(g^{k\bar l}(\sigma_t)_{k\bar l})_{\ijb}+(f_t)_{\ijb}\\
   =&(g^{k\bar l})_\ijb(\sigma_t)_{k\bar l}+(g^{k\bar l})_i((\sigma_t)_{k\bar l})_{\bar j}+(g^{k\bar l})_{\bar j}((\sigma_t)_{k\bar l})_{i}+(g^{k\bar l})((\sigma_t)_{k\bar l})_\ijb+(f_t)_{\ijb}\\
   \end{split}
\end{equation*}
Hence by \eqref{F-bound-e1}, we have
\begin{equation}\label{F-bound-e2}
|F_\ijb|\le C_3(|\Rm|+1).
\end{equation}

On the other hand, by
 Corollary \ref{eqivalent1}, Lemma \ref{3rdorder1-l1}, \eqref{3rdorder2}, \eqref{curvature1} and \eqref{Q13} we have:
\begin{equation}\label{S5}
\begin{split}
(\Delta - \frac{\p }{\p t})Q\ge C(n) | \Rm|^2-C_4.
\end{split}
\end{equation}
 By \eqref{f-e5}, (\ref{S4}), (\ref{S5}), \eqref{F-bound-e2}, Lemma \ref{3rdorder1-l1}, Corollary \ref{eqivalent1} and Corollary \ref{gradientf}, we can find positive constants $C_5$, $C_6$, $C_7$ depending only on the quantities in the lemma such that in $M\times[0,T]$
\begin{equation}\label{S6}
\begin{split}
( \Delta - \frac{\p }{\p t})&\lf[(1+S)^\frac12 +C_5|  \nabla
w|^2+C_6Q\ri]\\
&\ge \lf[(1+S)^\frac12 +C_5|  \nabla w|^2+C_6Q\ri]
-C_7
\end{split}
\end{equation}
where we have used \eqref{f-e5}. By assumption, for each $t\in [0, T)$, $S$ is bounded. One may then proceed as in the proof of Lemma \ref{Deltabound1} to conclude that $$
S\le C
$$
for some constant $C$ depending only on the constants mentioned in the lemma.
\end{proof}

From the above Lemmas we may conclude
\begin{cor}\label{holdernormestimate}
 Let $v$ be as in Lemma \ref{S-e6}.   There is a constant $C$ depending only on the quantities in Lemma \ref{S-e6}  such that

 $$\|v\|_{2+\alpha, 1+\alpha/2, M\times[0, T)}\leq C$$.
 \end{cor}
\begin{rem}\label{s3r1}  Let $\Omega$ be a bounded domain. Suppose the quantities that we want to estimates in this section are bounded in $\p\Omega\times[0,T)$.  Then it is easy to see that from the above proofs we may conclude that the quantities are also bounded in $\Omega\times[0,T)$.

\end{rem}

 \section{Applications }\vskip .2cm

We will now apply the results of the previous sections.  Given a solution $v(x, t)$ to \eqref{MA1} as in the previous section, we are interested in establishing conditions under which we have longtime existence of $v(x, t)$ as a solution to  \eqref{MA1}.  More generally we are also interested in where singularities can form when  \eqref{MA1} does not admit a longtime solution.  In Theorem \ref{singularities} below we describe where singularities of \eqref{MA1} can occur in the case where our solution corresponds to a solution to the \KR flow.  As a corollary we will establish a longtime existence result for the \KR flow in Corollary \ref{longtimeexistence} which improves the longtime existence result in \cite{ChauTamYu08}.  Then in Theorem \ref{mainthm}, combining the $C^0$ estimate in \cite{ChauTam09} with the a priori estimates of the previous section we establish a longtime convergent solution to \eqref{MA1} under certain conditions which generalize those in the main result in  \cite{ChauTam09}.

  The following Theorem describes where singularities can or cannot form under the \KR flow in terms of the existence of plurisubharmonic functions defined on subsets of a complete non-compact \K manifold.

\begin{thm}\label{singularities}
Let $(M, g_0)$ be a complete non-compat \K manifold such that
\begin{enumerate}
\item [i)] $|Rm_0(x)| \to 0$ as $d(x)\to \infty$ where $d(x)$ is the distance function on $M$ from some $p\in M$.
\item [ii)] The injectivity radius of $(M, g_0)$ is uniformly bounded below by some $c>0$.
\item[iii)] There exists an open set $S$ with smooth compact boundary $\partial{S}$ and a smooth function $F$ which is strictly pluri-subharmonic on $S$ and smooth up to $\partial S$.
 \end{enumerate}
 Let $g(t)$ be a solution to the \KR flow $g'=-Rc$ on $M\times[0, T)$ with initial condition $g(0)=g_0$.   Then for any closed set $N$ contained in $S$, the Riemannian curvature tensor of $g(t)$ and all its covariant derivatives are bounded on $N\times[0, T)$ provided they are bounded on $\partial S\times[0, T)$.
 \end{thm}

 \begin{proof} By the results in \cite{Sh2}, we may assume that $g_0$ is smooth and
 that all the covariant derivatives of the curvature tensor
 of $\omega_0$ are bounded on $M$.  Now suppose $g(t)$ solves  the \KR flow $g'=-Rc$ on $M\times[0, T)$ with initial condition $g(0)=g_0$, and that the curvature of $g(t)$ all its covariant derivatives are bounded on $\partial S\times[0, T)$.

   By the discussion in the introduction, we know that there is a solution $v(x, t)$ to
  \begin{equation}\label{MAa}
\left\{%
\begin{array}{ll}
    \dfrac{\p v}{\p t}= \log \dfrac{(\sigma+ \ii\partial\bar{\partial}v)^n}{(\omega_0)^n}      \\
   v(x,0)=0. \\
\end{array}%
\right.
\end{equation}
on $M\times[0, T)$ such that $\sigma=-tRc_0 + \omega_0$ and $\omega=\sigma+ \sqrt{-1}\partial\bar{\partial}v$ where $\omega$ is the \K form for $g(t)$. By Theorem 9.1 in \cite{ChauTamYu08} we can find   sufficiently large bounded open sets $\Omega_1$, $\Omega$ with smooth boundary such that

\begin{itemize}
\item  [(i)]  $\partial S \subset \Omega_1\subset\subset \Omega$;
\item [(ii)] $|\nabla_t^k Rm(x, t)|\leq C_k$ on $\Omega^c_1\times [0, T)$ for all $k$ and some $C_k$ depending only on $\Omega$ and $k$.
\end{itemize}
  In particular, from our hypothesis and the definition of $\sigma$ it is not hard to show we thus have $|\nabla_0^k v|\leq C_k$ on $\partial (\Omega\bigcap S) \times [0, T)$ for all $k$ and some $C_k$ depending only on $k$ and $T$.   Now if $N$ is any closed subset of $S$, then we have $N=(\Omega_1\bigcap N)\bigcup (\Omega_1^c \bigcap N)$ where $\Omega_1\bigcap N$ has compact closure in $\Omega\bigcap S$.  Hence by condition (ii) above, we see  that to prove the theorem it will be sufficient to prove that  $|\nabla_0^k v|\leq C_k$ on $\Omega' \times [0, T)$ for all $k$ and some $C_k$ depending only on $g_0$, $k$, $T$ and $\Omega'$ where $\Omega'$ is any closed subset of $\Omega\bigcap S$.  We now proceed to do this.  We begin by showing

 \begin{bf} Claim 1: \end{bf}  There exists $C>0$ such that $\sup_{(\Omega\bigcap S) \times [0, T)}|v_t| \leq C$.
\vspace{12pt}

We will establish the claim by using the a priori estimates from the previous section applied to domains (see Remark \ref{s3r1}), and by using the plurisubharmonic function $F$ and arguing as in \cite{LZ} where the authors considered the normalized \KR flow $g'=-Rc-g$ on a complete non-compact \K manifold (also see \cite{TZ} and \cite{Ts}).

We begin by differentiating  \eqref{MAa}  with respect to $t$
successively to obtain the following which we express in an
orthonormal coordinate with respect to $g(t)$:
   \begin{equation}\label{MAat}
 \dfrac{\p v_{t}}{\p t}= \Delta v_{t} +(\sigma')_{i\bar{i}}
 \end{equation}

  \begin{equation}\label{MAatt}
  \begin{split}
 \dfrac{\p v_{tt}}{\p t}= &\Delta v_{tt} - (g'(t))_{i\bar{j}}(v_t)_{i\bar{j}}- (g'(t))_{i\bar{j}}(\sigma')_{i\bar{j}}\\
 =&\Delta v_{tt} - (g'(t))_{i\bar{j}}(g'(t))_{i\bar{j}}\\
 \leq&\Delta v_{tt}\\
\end{split}
 \end{equation}
 Now from our above observation and \eqref{MAa}, we know
 that $|v_{tt}|$ is  bounded on $\partial(\Omega\bigcap S)\times [0, T)$.  From this fact and \eqref{MAatt}, we conclude by the maximum principle that $v_{tt}$, and thus $v_t$ and $v$ is  bounded from above on $(\Omega\bigcap S)\times [0, T)$.  Now for the bound from below, we use the pluri-subharmonic function $F$ and compute  in an orthonormal coordinate with respect to $g(t)$:

   \begin{equation}\label{MAatbelow}
  \begin{split}
 \dfrac{\p}{\p t}((2T-t)v_t+v-F )= &\Delta ((2T-t)v_t+v-F )\\
 &-\Delta u +\Delta F + 2T(\sigma'_t)_{i\bar{i}}- t(\sigma'_t)_{i\bar{i}}\\
 =&\Delta ((2T-t)v_t+v-F ) - n+
 (2T\sigma)_{i\bar i}+F_{i\bar{i}}\\
 \geq&\Delta ((2T-t)v_t+v-F ) -n\\
\end{split}
 \end{equation}
 where in the second equality we have used the fact that by scaling $F$ we may assume that we have $2T(\sigma'_t)_{i\bar{i}}+F_{i\bar{i}}\geq0$ on $(\Omega\bigcap S)$, $\sigma'=-Rc_0$ and the fact that $n=g^{i\bar{j}}g_{i\bar{j}}=g^{i\bar{j}}(\sigma+\partial\bar{\partial}v)_{i\bar{j}}=\Delta v -t g^{i\bar{j}}(Rc_0)_{i\bar{j}}+g^{i\bar{j}}(\omega_0)_{i\bar{j}}$.  Arguing as above and using the bound on $v$ from above, we conclude from \eqref{MAatbelow} and the maximum principle that $v_t$ and thus $v$ is bounded from below on $(\Omega\bigcap S)\times [0, T)$.  This completes the proof of the claim.

We now modify \eqref{MAa} so that the varying background metric is uniformly \K. We note that
   \begin{equation}
 \sigma=\dfrac{T-t}{T}\sigma(0)+\dfrac{t}{T}\sigma(T)
  \end{equation}
and we let
   \begin{equation}
   \begin{split}
\hat{ \sigma}&=\dfrac{T-t}{T}\sigma(0)+
\dfrac{t}{T}(\sigma(T) +\partial\bar{\partial}F)\\
\hat{v}&=v-\dfrac{t}{T}F.
\end{split}
\end{equation}
By condition (iii), by scaling $F$ we may
assume  that $\sigma(T) +   \partial\bar{\partial}F$   is
equivalent to $\omega_0$  on $(\Omega\bigcap S)$. Hence $\hat \sigma$
is uniformly equivalent to $\omega_0$ on
$(\Omega\bigcap S)\times[0,T)$.

Then from the above equations and \eqref{MAa} we have
  \begin{equation}\label{MAamodified}
\left\{%
\begin{array}{ll}
    \dfrac{\p \hat{v}}{\p t}= \log \dfrac{(\hat{\sigma}+ \partial\bar{\partial}\hat{v})^n}{(\omega_0)^n}- \dfrac{F}{T}\      \\
   \hat{v}(x,0)=0. \\
\end{array}%
\right.
\end{equation}
The point is that the background $\hat{\sigma}$ is now a \K metric which is uniformly equivalent to $\omega_0$ on $(\Omega\bigcap S)\times [0, T)$.  On the other hand, our previous estimates imply that $|\hat{v}_t|$ and thus $|\hat{v}|$ are uniformly bounded on $(\Omega\bigcap S)\times [0, T)$.  From this fact, the above observation that $|\nabla_0^k v|$ and thus $|\nabla_0^k \hat{v}|$ is uniformly bounded on $\partial(\Omega\bigcap S)\times [0, T)$ for all $k$, and the estimates in the previous section we can conclude by the maximum principle that $\hat{v}_{i\jbar}, \hat{v}_{i\jbar k}, \hat{v}_{ti\jbar}$ are uniformly bounded on $(\Omega\bigcap S) \times [0, T)$.  Thus by differentiating \eqref{MAamodified} and applying parabolic Schauder estimates we conclude that $|\nabla_0^k \hat{v}|$ and thus $|\nabla_0^k v|$ is uniformly bounded on $\Omega'\times [0, T)$ for all $k=0, 1, 2,..$ by a constant depending only on $T$, $g_0$, $\Omega'$ and $k$ where $\Omega'$ is any closed set in $(\Omega\bigcap S)$.   This completes the proof of the Theorem.
\end{proof}

As a Corollary of Theorem \ref{singularities} we have the following longtime existence result for the \KR flow.

 \begin{cor}\label{longtimeexistence}
Let $(M, g_0)$ be a complete non-compat \K manifold such that
\begin{enumerate}
\item [i)] $|Rm_0(x)| \to 0$ as $d(x)\to \infty$ where $d(x)$ is the distance function on $M$ from some $p\in M$.
\item [ii)] The injectivity radius of $(M, g_0)$ is uniformly bounded below by some $c>0$.
\item[iii)] There exists a strictly pluri-subharmonic function $F$ on $M$.
 \end{enumerate}
 Then the \KR flow $g'=-Rc$ has a longtime solution $g(t)$ on $M$ with initial condition $g(0)=g_0$.
 \end{cor}

 We now turn to our second application of the a priori estimates from our previous section.  We will establish the following convergence result for \eqref{MA1} which generalizes the main convergence result in \cite{ChauTam09}.
 \begin{thm}\label{mainthm} Let $(M^n,g_0)$ be a smooth complete non-compact
\K manifold with $n\ge3$, and let $\Omega$ be a (1,1)-form on $M$ such that $\Ric_0-\Omega=\ii\p\bar\p f_0$  for some smooth potential  $f_0$.  Suppose we have $\sum_{k=0}^1|\nabla_0^kRm|+\sum_{k=0}^3|\nabla_0^k f| < \infty.$  Then we have the following:
\begin{enumerate}
\item [1.] The following modified \KR flow has a long time smooth solution $g(t)$.

\begin{equation}\label{MKRF}
    \left\{%
\begin{array}{ll}
    \dfrac{\p   {g}_{i\jbar}}{\p t} =- R_{i\jbar} + \Omega_{i\jbar}
      \\
     g_\ijb(x, 0) =(g_0)_\ijb \\
\end{array}%
\right.
\end{equation}

\item [2.]  If in addition the potential  $f_0$ satisfies

\begin{enumerate}
       \item [(a)]         \begin{equation}\label{decaycondition}
        |f_0|(x)\le \frac{C_1}{1+\rho_0^{2+\e}(x)}
\end{equation}
   for some $C_1, \e >0$,  and all $x\in M$ where
    $\rho_0(x)$ is the distance function from a fixed $o\in M$.
        \item [(b)] The following Sobolev inequality is true:
    \begin{equation}\label{Sobolev}
\lf(\int_M|\phi|^{\frac{2n}{n-1}}dV_0\ri)^{\frac{n-1}{n}}\le
C_2\int_M|\nabla_0 \phi|^2dV_0
\end{equation}
 for some $C_2>0$ and all $\phi\in C_0^\infty(M)$.
\item[(c)] There exists a constant $C_3>0$ such that
\begin{equation}\label{volumegrowth}
    V_0(r)\le C_3r^{2n}
\end{equation}
for some $C_3 >0$ and all $r$ where $V_0(r)$ is the volume of the geodesic
ball with radius $r$  centered at some $o\in M$.
\end{enumerate}
Then as $t\to\infty$, $g(t)$ converges
uniformly on compact sets in the $C^\infty$ topology on $M$ to a complete
 \KR metric $g_{\infty}$ on $M$ which is uniformly equivalent to $g_0$, has bounded geometry of order $2+\alpha$ and  has Ricci form equal to $\Omega$.
\end{enumerate}

\end{thm}

\begin{rem} If we assume further that  all the covariant derivatives of $Rm_0$ and $f_0$ are bounded, then $g_{\infty}$ will also have all covariant derivatives of curvature bounded.  \end{rem}

As in \cite{ChauTam09} our approach is to consider \eqref{MA1} where we set $\sigma=g_0$.  In other words, we consider the equation
                    \begin{equation}\label{MA2}
\left\{%
\begin{array}{ll}
    \dfrac{\p v}{\p t}= \log \dfrac{\det ((g_0)_{k \bar{l}}+ v_{k
\bar{l}})}{\det ((g_0)_{k \bar{l}})}-f_0
      \\
    v(x,0)=0 \\
\end{array}%
\right.
\end{equation}
A straight forward calculation show that if $v(x, t)$ solves \eqref{MA2} on $M\times[0, T)$ then $g_\ijb(t)=(g_0)_{\ijb}+v_\ijb$ is a family of \K metrics on $M$ which solves \eqref{MKRF} on $M\times[0, T)$.  On the other hand, it is not hard to show that given a solution $g_\ijb(t)$ to \eqref{MKRF} on $M\times[0, T)$, then we obtain a solution $v(x, t)$ to \eqref{MA2} on $M\times[0, T)$ (see  \cite{ChauTam09}).

We now prove the first part of the Theorem. First note that the assumptions on the curvature tensor imply $(M,g_0)$ has bounded geometry of order $2+\alpha$ for some $0<\alpha<1$ by \cite{TY3}. The assumptions on $f$ then imply that $f\in C^{2+\alpha}(M)$. Hence under the hypothesis in the theorem, by Proposition 2.1 there exists a maximal smooth solution $v$ to \eqref{MA2} satisfying the conclusions of the Proposition on $M\times[0, T)$ for some $T>0$.  On the other hand, differentiating \eqref{MA2} with respect to $t$ gives

   \begin{equation}
   \left\{%
\begin{array}{ll}
    \dfrac{\p v_t}{\p t}= \Delta_t v_t
      \\
    v_t(x,0)=f_0 \\
\end{array}%
\right.
\end{equation}
 and thus by the maximum principle and our hypothesis on $f_0$, we conclude that $|v_t|$ and thus $|v|$ is uniformly bounded on $M\times[0,T)$.  Now by Remark 1, we may have that the conclusion of Corollary 3.3 is true on $M\times[0,T)$ for some constant $C$.  From this, Remark 1 and Proposition 2.1 we see that if $T<\infty$ we could then extend $v$ as a solution to \eqref{MA2} to $M\times[0, T')$ for some $T'>T$ which contradicts the maximality of $T$.  Thus we must have $T=\infty$ which establishes the first part of the Theorem.

We now prove the second part of the Theorem on convergence.  In \cite{ChauTam09} a $C^0$ estimate was established for \eqref{MA2} under the conditions of Theorem \ref{mainthm} and the additional assumption that $\Omega=0$.  The key difference in this case is that  the corresponding metrics $g(t)$ will evolve under the standard \KR flow, and general \KR theory may then be applied.  On the other hand, the a priori estimates of the previous section basically ensure that by the same proof as in \cite{ChauTam09} we obtain the same $C^0$ estimate without this additional assumption.   Our a priori estimates and Proposition \ref{MA-shorttime-l1} then imply the existence of a longtime solution to  \eqref{MA2} which stays uniformly bounded in the $C^0$ norm on $M$ with additional higher order derivative bounds aswell.  We state this more precisely in the following

\begin{lem}\label{regularity-l1}
Let $(M^n,g_0)$ be as in Theorem \ref{mainthm}.  Then \eqref{MA2} has a smooth solution  $v$ on $M\times[0, \infty)$ such that for each $0\leq l \leq 4$, $\| \nabla_0^l v (x, t)\|_{g_0}$ is bounded by a constant depending  only on $l$, $g_0$, $f_0$.

  Moreover, given any $4< l <\infty$ and a compact set $S\subset M$ then
  $\| \nabla_0^l v (x, t)\|_{S, g_0}$ is bounded by a constant depending only on $S$, $l$, $g_0$, $f_0$.

\end{lem}
\begin{proof}
By Proposition \ref{MA-shorttime-l1}, there exists a solution $v(x, t)$ to  \eqref{MA2} on $M\times[0, T)$ for some $T>0$ satisfying the properties in the Proposition.  Thus $g_\ijb(t)=(g_0)_{\ijb}+v_\ijb$ is a family of \K metrics on $M$ such that for each $t\in [0, T)$, $g(t)$ is equivalent to $g_0$.  Moreover, by remark \ref{remfourthorder} we in fact have that for each $t\in [0, T)$, $g(t)$ has bounded curvature.  It follows from the proof of Lemma 4 in \cite{ChauTam09} that we have $\sup_{t\in[0, T)} |v(t)|\leq C$ for some $C$ independent of $T$.

 In particular, by Corollary \ref{holdernormestimate} we have $\|v\|_{C^{2+\alpha, 1+\alpha/2}(M\times[0, T))}$ bounded independent of $T$.  Thus by considering the pull back of \eqref{MA1} in an arbitrary quasi-coordinate and applying a bootstrapping argument, as in the last part of the proof of Proposition  \ref{MA-shorttime-l1}, we conclude that $\|v(t)\|_{C^{4+\alpha}(M)}$ is bounded independent of $t\in[0, T)$ (see Remark \ref{remfourthorder}).  Thus by Proposition \ref{MA-shorttime-l1} we can extend $v(x, t)$ as a solution to \eqref{MA2} past $T$,  and we may then assume that  $T=\infty$ in the above discussion.  This completes the proof of the first part of the Lemma.  The proof of the second part of the Lemma follows from iterating the above bootstrapping argument.

 \end{proof}

Let $v(x, t)$ is a longtime solution to \eqref{MA2} as in Lemma \ref{regularity-l1}.  We want investigate the longtime behavior of $w(x, t)=u_t(x, t)$.



 \begin{lem}\label{f-convergence-l1}
$w\to 0$ pointwise on $M$ as $t\to\infty$.
 \end{lem}
 \begin{proof} We begin by showing that $|\nabla w^k|\to0$ as $t\to\infty$ for some integer $k\ge1$.
By (15) of \cite{ChauTam09}, if $p=2k+2$ with $k$ being a large integer, we have
 \begin{equation}\label{f-convergence-e1}
 \frac{d}{dt}\int_M w^pdV_t\le -C\int_M |\nabla w^k|^2dV_t
 \end{equation}

 where $C$ is a positive constant independent of $t$. Let $x_0\in M$ suppose there exist  $t_i\to\infty$ and $\e>0$ such that   $|\nabla w^k|(x_0,t_i)\ge \e$.  On the other hand, there is a neighborhood $U$ of   $x_0$, such that
 \begin{equation}
 |u_{t\alpha\beta}|+|u_{tt\alpha}|\le C
 \end{equation}
 in $U$, where $\alpha$ etc. denote  the indices for the real coordinates. Hence there is $\delta>0$ such that for all $(x,t)\in B_0(x_0)\times [t_i,t_i+\delta]$ we have $|\nabla w^k|\ge \frac\e2.$ In particular,
 $$
 \int_M|\nabla w^k|^2dV_t\ge \e'
 $$
 for all $t\in [t_i,t_i+\delta]$ for some $\e'>0$ independent of $i$. This is impossible because of \eqref{f-convergence-e1}.  Thus we have established that $|\nabla w^k|\to0$ as $t\to\infty$ for some integer $k\ge1$.  On the other hand, \eqref{f-convergence-e1} shows that the integral  $\int_M w^pdV_t$ is uniformly bounded for all $t\geq0$. Combining these last two facts together with Lemma \ref{regularity-l1}, we conclude that $w\to 0$ pointwise on $M$ as $t\to\infty$.
 \end{proof}

 We now combine the above lemma's with a Liouville
 theorem of Yau \cite{Y} for $L^p$ harmonic functions on
 complete Riemannian manifolds to give a proof of
 Theorem \ref{mainthm}.
\begin{proof}[Proof of Theorem \ref{mainthm}]
By Lemma \ref{regularity-l1}, given any sequence $ t_i\to \infty$ there is a subsequence of $u(x, t_i)$  converging smoothly and uniformly on compact subsets of $M$.  To prove Theorem \ref{mainthm} it suffices to prove that such a limit is independent of the sequence $t_i$ which we now proceed to do.

 Suppose $u(x, t_i)$ converges to $u_1(x)$ and $u(x, s_i)$ converges to $u_2(x)$ smoothly and uniformly on compact subsets of $M$, where $t_i, s_i \to \infty$ and $s_i \leq t_i$ for every $i$.  We claim  that $v=u_1 -u_2$ satisfies the Laplace equation \begin{equation}\label{laplace}\Delta_h v=0\end{equation} on $M$ where $h$ is a complete \K metric on $M$ which is equivalent to $g_0$ and has all covariant derivatives of its curvature tensor bounded.  Indeed, by \eqref{MA2} we have
\begin{equation}\label{interpolation}
\begin{split}
f(s_i)-f(t_i)&=\log \det ((g_0)_{k \bar{l}}+ u_{k
\bar{l}}(t_i))-\log \det ((g_0)_{k \bar{l}}+ u_{k
\bar{l}}(s_i))\\
&=\int_{0}^{1} \frac{d}{ds} \log \det ((g_0)_{k \bar{l}}+ (su(t_i)-(1-s)u(s_i))_{k
\bar{l}})ds\\
&=(\int_{0}^{1} G^{l\bar{k}}(s) ds) (u(t_i)-u(s_i))_{l\bar{k}}
\end{split}
\end{equation}
 for each $i$, where $G^{i\jbar}(s)$ is the inverse of the \K metric $G_{i\jbar}(s)=(g_0)_{k \bar{l}}+ (su(t_i)-(1-s)u(s_i))_{k\bar{l}})$.  It follows from Lemma \ref{regularity-l1} that some subsequence of $(\int_{0}^{1} G^{l\bar{k}}(s) ds)$ converges smoothly and uniformly on compact subsets of $M$ to a smooth limit $h^{i\jbar}$ which is the inverse of a \K metric $h$ as above.  Our claim follows by taking a limit of \eqref{interpolation} and using Lemma \ref{f-convergence-l1}.  On the other hand, by the proof of Lemma 3 in \cite{ChauTam09} we know that $\int_M |u(x, t)|^p dV_0 \leq C$ for some $p$ and some $C$ independent of $t$.  Thus we have  $\int_M |v(x)|^p dV_0 \leq C$.   It follows from this, \eqref{laplace} and the Liouville theorems in \cite{Y1} that $v=0$ and thus $u_1=u_2$.  This completes the proof of Theorem \ref{mainthm}.

\end{proof}

\section{Appendix}
We begin by explicitly defining the local \H norms used in the definitions of the elliptic and parabolic \H spaces on $M$ used in \S 2 (Also see \cite{Kr}).  Let $\Omega$ be an open set in $\R^m$. Let $k\ge0$ be an integer and $0<\alpha<1$, then the $C^{k+\alpha}$ norm of a function $u$ on $\Omega$ is defined as:
 \begin{equation*}
 \begin{split}
 ||u||_{k+\alpha;\Omega}=\sum_{|s|=0}^k\sup_{\Omega}\lf| \p^s u \ri|
+\sum_{|s|=k}\sup_{x\neq x'\in \Omega} \frac{\lf| \p^s u(x)-\p^s u(x')\ri|}{|x-x'|^\alpha}\\
 \end{split}
\end{equation*}
 where $s$ is a multi-index, and $$\p^s u=\frac{\p^{|s|}u}{\p x^s}.$$ For $T>0$, the $C^{2k+\alpha, k+\frac\alpha2}$ norm on $\Omega_T=\Omega\times[0,T]$   is defined as:
 \begin{equation*}
 \begin{split}
 ||u||_{k+\alpha,k+\frac\alpha2;\Omega_T}=&\sum_{|s|+2r=0}^{2k}\sup_{\Omega}\lf| \p^r_t\p^s u \ri|\\
 &+\sum_{|s|+2r=2k}\sup_{(x,t)\neq (x',t')\in \Omega_T} \frac{\lf| \p^r_t\p^s u(x,t)-\p^r_t\p^s u(x',t')\ri|}{|x-x'|^\alpha+|t-t'|^\frac\alpha2}
  \end{split}
\end{equation*}
 where
 $$\p^r_t\p^s u=\frac{\p^{r+|s|}u}{\p t^r\p x^s}.$$ If there is no confusion, we will simply write $||u||_{k+\alpha}$ or $||u||_{k+\alpha,k+\frac\alpha2}$.

  Now suppose $(M, g)$ has bounded geometry of order $k+\alpha$ with respect to a quasi-coordinate system $\mathcal{F}$ where $k$ is even.  Fix some $T>0$ and consider the associated  parabolic \H norm $||\cdot ||_{k, \alpha; \mathcal{F}}$ for functions on $M\times[0, T)$ as in Definition \ref{Holderspace} where the additional subscript denotes  dependence on the quasi-coordinate $\mathcal{F}$.  The following lemma basically says that the associated Banach space $C^{k+\alpha, k/2+\alpha/2}(M\times[0, T))$ from definition \ref{Holderspace} is independent of the choice of quasi-coordinate.

 \begin{lem}\label{equivalent-l1}Let
 $\mathcal{G}=\{(\theta_p, V_p)|\ p\in M\}$ be another  family of quasi-coordinate neighborhoods with data $ K_1, K_2, R$.
 Then there is a constant $C>0$ such that for all smooth function $f$ on $M\times[0,T]$
 $$
||f||_{k, \alpha; \mathcal{G}}\le C ||f||_{k, \alpha; \mathcal{F}}
$$
\end{lem}
\begin{proof} We just prove the case that $k=2$. Let $f$ be a smooth function on $M\times[0,T]$. Let $(\theta_p,V_p)$ be a quasi-coordinate neighborhood in $\mathcal{G}$. Let $w_0\in D(R)$ and let $\theta_p(w_0)=q$. Let $(\xi_q,U_q)$ be a quasi-coordinate in $\mathcal{F}$ such that $\xi_q(0)=q$. By \cite[Lemma 3.2]{ChauTam07} in \cite{ChauTam07}, suppose $D(w_0,\e)=\{|w-w_0|<\e\}\subset D(R)$, then there is a local biholomorphism $\phi: D(w_0,\e)\to D(r)$ such that $\phi(w_0)=0$ and $\xi_q\circ\phi=\theta_p$, provided $\e<\e_0$ which depends only on $\mathcal{F}$ and $\mathcal{G}$. Suppose $\phi(w)=z$. Then
$$
\xi_q^*(f)=\theta_p^*(f)
$$
and
$$s_{i\jbar}\frac{\p z_i}{\p w_k} \ol{\frac{\p   z_j}{\p w_l}}.=h_{k\bar{l}}$$
on $\theta(D(w_0,\e))$ where $h_{k\bar{l}}:=\theta_q^*(g)(\frac{\p}{\p w_k},\frac{\p}{\p \bar w_l})$ and $s_{i\jbar}:=\xi_p^*(g)(\frac{\p}{\p z_k},\frac{\p}{\p \bar z_l})$.

Let $k=l$ and using the fact that $\xi_q^*(g)$ and $\theta_p^*(g)$ are uniformly equivalent to the Euclidean metric, we conclude that

$$
|\frac{\p z_i}{\p w_k}|\le C
$$
on $ D(w_0,\e)$. Now
$$
\frac{\p \theta_p^*(f)}{\p w_k}=\frac{\p \xi_q^*(f)}{\p z_i}\frac{\p z_i}{\p w_k}
$$
we conclude that on $ D(w_0,\e)$
$$
|\theta_p^*(f)|+|D_w  \theta_p^*(f)|\le C||f||_{2,\alpha,\mathcal{F}}.
$$
It is also easy to see that on $ D(w_0,\e)$
$$
 |D_t  \theta_p^*(f)|=|D_t  \xi_q^*(f)|\le C||f||_{2,\alpha,\mathcal{F}}.
$$

On the other hand,
\begin{equation}\label{2nd-e1}
\begin{split}
\frac{\p }{\p w_a}h_{i\jbar}=&\frac{\p }{\p z_b}s_{k\bar{l}}\frac{\p z_b}{\p w_a}\frac{\p z_k}{\p w_i} \ol{\frac{\p z_l}{\p   w_j}}\\
&+s_{k\bar{l}}\frac{\p^2 z_k}{\p w_i\p w_a} \ol{\frac{\p z_l}{\p   w_j}}.
\end{split}
\end{equation}
Consider the vector $v=\frac{\p}{\p z_k}\frac{\p^2 z_k}{\p w_i\p w_a}$. Let $v=a_j \phi_*\frac{\p}{\p w_j}.$
$$
\xi_q^*(g)(v,\phi_*\frac{\p}{\p \bar w_c})=a_jh_{j\bar{c}}.
$$
Hence
\begin{equation}\label{2nd-e2}
\begin{split}
   v=&h^{j\bar c}\xi_q^*(g)(v,\phi_*\frac{\p}{\p \bar w_c})\phi_*\frac{\p}{\p w_j}\\
   =&h^{j\bar c}s_{k\bar{l}}\frac{\p^2 z_k}{\p w_i\p w_a}\ol{\frac{\p z_l}{\p w_c}}\frac{\p z_d}{\p w_j}\frac{\p}{\p z_d}.
   \end{split}
\end{equation}
So
\begin{equation}\label{2nd-e3}
\begin{split}
\frac{\p^2 z_d}{\p w_i\p w_a}&=h^{j\bar c}s_{k\bar{l}}\frac{\p^2 z_k}{\p w_i\p w_a}\ol{\frac{\p z_l}{\p w_c}}\frac{\p z_d}{\p w_j}\\
&=h^{j\bar c}(\frac{\p }{\p w_a}h_{i\jbar}-\frac{\p }{\p z_b}s_{k\bar{l}}\frac{\p z_b}{\p w_a}\frac{\p z_k}{\p w_i} \ol{\frac{\p z_l}{\p   w_j}})\frac{\p z_d}{\p w_j}
\end{split}
\end{equation}
Hence we have
$$
|\frac{\p^2 z_d}{\p w_i\p w_a}|\le C
$$
on $D(w_0,\e)$. In fact, from (0.3) we see that $$\|\frac{\p^3 z_d}{\p w_i\p w_a\p w_j}\|_{\alpha}\leq C$$ ion $D(w_0,\e)$ (Note that since $s_{i\jbar}(z)$ is $C^{2+\alpha}$ and $z(w)$ is $C^1$, it follows that $\frac{\p^2 }{\p z^2}s_{i\jbar}(z(w))$ is $C^{\alpha}$ with respect to $w$).  In particular, it is easy to see that
on $D(w_0,\e)$
$$
 |D_w^2 \theta_p^*(f)|\le C||f||_{2,\alpha,\mathcal{F}}.
$$
Now if $w\in D(w_0,\e)$, then

\begin{equation}\label{2nd-e3}
\begin{split}
\frac{\p \theta_p^*(f)}{\p w_k\p w_l}(w)&-\frac{\p \theta_p^*(f)}{\p w_k\p w_l}(w_0)\\
=&\frac{\p \xi_q^*(f)}{\p z_i\p z_j}\frac{\p z_i}{\p w_k}\frac{\p z_j}{\p w_l}(w)+\frac{\p \xi_q^*(f)}{\p z_i}\frac{\p^2 z_i}{\p w_k\p w_l}(w)\\
&-\frac{\p \xi_q^*(f)}{\p z_i\p z_j}\frac{\p z_i}{\p w_k}\frac{\p z_j}{\p w_l}(w_0)-\frac{\p \xi_q^*(f)}{\p z_i}\frac{\p^2 z_i}{\p w_k\p w_l}(w_0).
\end{split}
\end{equation}
Since $|z-z_0|\le C|w-w_0|$, we can conclude that
$$
\frac{|\frac{\p \theta_p^*(f)}{\p w_k\p w_l}(w) -\frac{\p \theta_p^*(f)}{\p w_k\p w_l}(w_0)|}{(|w-w_0|^2+|t-t_0|)^\frac\alpha2}\le C||f||_{2,\alpha,\mathcal{F}}
$$
for $w\in D(w_0,\e)$. One can also obtain the H\"older estimate for $\p_t\theta_p^*(f)$. This completes the proof of the Lemma.

\end{proof}

\bibliographystyle{amsplain}

 \end{document}